\documentclass[12pt]{article}
\usepackage{hyperref}
\usepackage{mathtools,amssymb, amsthm,bm,latexsym,todonotes}
\usepackage{verbatim,comment}
\newtheorem{theorem}{Theorem}
\newtheorem{lemma}[theorem]{Lemma}
\newtheorem{corollary}[theorem]{Corollary}

\newtheorem{definition}[theorem]{Definition}

\newtheorem{problem}{Open Problem}
 \def\fn{\mbox{Fn}}
 
 \def\dom{\mbox{dom}}

\def\fc{{\mbox{\scriptsize fast}}}

\def\force{\Vdash}
\def\supp{\mbox{supp}}
\def\hod{\mathop{\mbox{HOD}}}

\def\P{{\mathcal P}}

\def\N{{\Bbb N}}
\def\bP{{\Bbb P}}
\def\bQ{{\Bbb Q}}
\def\bR{{\Bbb R}}
\def\R{{\Bbb R}}
\def\C{{\Bbb C}}
\def\bDelta{\mathbf{\Delta}}
\def\bSigma{\mathbf{\Sigma}}

\DeclarePairedDelimiter{\abs}{\lvert}{\rvert}
\DeclarePairedDelimiter{\set}{\lbrace}{\rbrace}

\DeclareMathOperator{\powset}{\mathcal P}
\DeclareMathOperator{\Th}{Th}

\DeclareMathOperator{\ran}{ran}

\DeclareMathOperator{\Coll}{Coll}
\newcommand{\Pmax}{{\bP_\text{max}}}

\newcommand{\isom}{\cong}

\DeclareMathOperator{\crit}{crit}

\newcommand{\limp}{\to}
\newcommand{\lequ}{\leftrightarrow}

\usepackage{lineno}
\begin{document}

\title{On the categoricity of complete second order theories\thanks{The first and second author would like to thank  the Academy of Finland, grant no: 322795. This project has received funding from the European Research Council (ERC) under the European Union’s Horizon 2020 research and innovation programme (grant agreement No 101020762).}}
\author{Tapio Saarinen\\
University of Helsinki \and Jouko V\"a\"an\"anen\\
University of Helsinki and \\
University of Amsterdam \and Hugh Woodin\\
Harvard University}
\date{}
\maketitle

\begin{abstract}
We show, assuming PD, that every complete finitely axiomatized second order theory with a countable model is categorical, but that there is, assuming again PD, a complete recursively axiomatized second order theory with a countable model which is non-categorical. We show that the existence of even very large (e.g. supercompact) cardinals  does not imply the categoricity of all finite complete second order theories. More exactly,  we show that a non-categorical complete finitely axiomatized second order theory can always be obtained by (set) forcing. We also show that the categoricity of all finite complete second order theories with a model of a certain singular cardinality $\kappa$ of uncountable cofinality can be forced over any model of set theory. Previously, Solovay had proved, assuming $V=L$, that every  complete finitely axiomatized second order theory (with or without a countable model) is categorical, and that in a generic extension of $L$ there is a complete  finitely axiomatized second order theory with a countable model which is non-categorical. 
\end{abstract}

\section{Introduction}
A second order theory $T$ is  \emph{complete} if it decides, in the semantical sense, every second order sentence $\phi$ in its own vocabulary  i.e.  if for every such $\phi$ either $T\models\phi$ or $T\models\neg\phi$, or equivalently, all models of $T$ are second order equivalent. 
The question we investigate in this paper is whether every complete second order theory is \emph{categorical} in the sense that all of its models are isomorphic. 
Already in 1928 Fraenkel \cite{zbMATH02575684}  mentions this question  as a question `calling for clarification'.   Carnap \cite{MR1768840} claimed a positive answer, but his proof had an error (see \cite{MR1951297}).

For mere cardinality reasons there are always complete non-categorical second order theories. One needs only  consider  models of the empty vocabulary. Since there are only continuum many different second order theories, there must be two such models of different cardinality with the same (\emph{a fortiori }complete) second order theory.   

Categoricity of complete second order theories would follow if all second order equivalent models were isomorphic, which is not the case again for cardinality reasons. However, if $V =L$, then \emph{countable} second order equivalent models are isomorphic \cite{Ajtai} and, moreover,  every complete finitely axiomatized 
second order theory is categorical \cite{Solo}. But if a  Cohen real is added to a model of $V=L$, then there are countable non-isomorphic second order equivalent  models \cite{Ajtai}, and if $\aleph_1$ Cohen-reals are added to a model of $V=L$,  there is a complete finitely axiomatized second order theory (with a countable model) which is non-categorical \cite{Solo}.

Fra\"iss\'e \cite{MR0035811,MR0042374} conjectured that  countable  second order equivalent ordinals are equal. Marek \cite{MR0381990,MR0381991} showed that Fra\"iss\'e's conjecture is true under the assumption $V=L$, and false in a forcing extension obtained by collapsing an inaccessible cardinal to $\omega_1$. 


The ambitious goal in the area of  this paper is to decide in a definitive way the status of categoricity of complete second order theories. Since we are dealing with a question that cannot be decided in ZFC alone,  it is natural to make an assumption such as PD, a consequence of the existence of large cardinals (e.g. infinitely many Woodin cardinals). We offer a partial solution to the  full question by solving the case of second order theories with countable models. We have also partial results about theories with uncountable models. In particular, we show that a non-categorical complete finitely axiomatized second order theory can always be obtained by (set) forcing. This shows that large cardinal assumptions cannot imply, as $V=L$ does, the categoricity of all complete finitely axiomatized second order theories.
\medskip

\noindent \emph{Notation:} We recall the usual definition of the beth hierarchy: $\beth_0=\omega$, $\beth_{\alpha+1}=2^{\beth_\alpha}$, and $\beth_\nu=\sup_{\alpha<\nu}\beth_\alpha$ for limit $\nu$. An ordinal $\alpha$ is called a \emph{beth fixed point} if $\alpha=\beth_{\alpha}$. If $\mu$ is a cardinal, we use $\fn(I,J,\mu)$ to denote the poset of partial functions $I\to J$ of cardinality $<\mu$, ordered by $p\le q\iff q\subseteq p$. The trivial poset $\fn(\emptyset,\emptyset,1)$ is denoted $(\{0\},=)$.

We denote the second order theory of a structure $M$ by $\Th_2(M)$. A second order theory $T$ is complete if $\Th_2(M) = \Th_2(N)$ for all $M, N \models T$, and $T$ is categorical if $M \isom N$ for all $M, N \models T$. For second order sentences $\phi, \psi$ we write $\phi \models \psi$ to mean $M \models \phi$ implies $M \models \psi$ for all $M$, and similarly $T \models T'$ for second order theories $T, T'$, and we say $T$ axiomatizes $T'$. If $T$ is a finite (resp. recursive) set of sentences and $T \models T'$, we say $T'$ is finitely (resp. recursively) axiomatizable. 

A cardinal $\lambda$ is second order characterizable if there is a second order sentence $\phi$ in the empty vocabulary such that $N \models \phi$ if and only if $\abs N = \lambda$.

\section{The case of $L[U]$}

It is already known that if $V=L$, then every complete finitely axiomatized second order theory is categorical \cite{Solo}. We now show that this also holds if $V = L[U]$, and we show there are complete recursively axiomatized second order theories that are non-categorical (with very large models).

Assuming $V=L[U]$, we write $\kappa$ for the sole measurable cardinal, $U$ for the unique normal measure on $\kappa$ and $<_{L[U]}$ for the canonical well-order. By a $L[U]$-premouse we mean a structure $(L_\alpha[W],\in,W)$ where $W$ is an $L_\alpha[W]$-ultrafilter on some $\gamma < \alpha$. Recall that a premouse $(L_\alpha[W], \in, W)$ is iterable (under taking iterated ultrapowers), i.e. that every iterate is well-founded, if and only if every iterate in an iteration of any countable length is well-founded. Observe that every iterate in an iteration of countable length has the same cardinality as the original premouse, so the iterability of a premouse is expressible in second order logic. See for example \cite[chapter 20]{HigherInf} for more details.

\begin{theorem}
Assume $V = L[U]$. Every complete finitely axiomatized second order theory is categorical.
\end{theorem}
\begin{proof}
Suppose $\phi$ is a complete second order sentence in a vocabulary with a single binary relation symbol $R$ (for simplicity). Note first that $\phi$ has models in only one cardinality. If not, let $N$ be a model of $\phi$ of least cardinality, and $M$ another model with $\abs M > \abs N$. Let $\theta$ be the sentence
\[
\exists P \exists R' (\theta'(P) \land \phi'(P, R'))
\]
where
\begin{itemize}
\item
$P$ is a unary predicate, not occurring in $\phi$, and $R'$ is a binary relation symbol not occurring in $\phi$.
\item $\phi'(P,R')$ is a modification of the sentence $\phi$, where the first order quantifiers $\exists x \dots$ are relativised to $P$ as $\exists x (P(x) \land {\dots})$, and each occurrence of $R$ is replaced by $R'$.
\item $\theta'(P)$ says that the cardinality of $P$ is smaller than the ambient domain of the model (for example, that there is no injective function with range contained in $P$).
\end{itemize}
As $\phi$ is complete and $M \models \theta$ (by taking $(P, R')$ isomorphic to $N$), also $N \models \theta$, so there is a model of $\phi$ of cardinality smaller than that of $N$, which is a contradiction. Thus all models of $\phi$ have the same cardinality.

Now let $M_0$ be the $<_{L[U]}$-least model of $\phi$. Suppose first that $\abs{M_0} > \kappa$: in this case we can mimic the categoricity argument for $L$ as follows. Let $\theta$ be the sentence
\[
\exists E \exists u \exists m \exists P \exists R'(\theta'(E,u) \land \phi'(P,R') \land \theta_{least}(E,u,m) \land \theta_{isom}(E,m,P,R'))
\]
where
\begin{itemize}
\item $E, R'$ are binary predicate symbols, $P$ a unary predicate symbol and $u, m$ are first order variables, none occurring in $\phi$ (the intuition is that $E$ is $\in$, $u$ is a normal ultrafilter, $m$ is a structure in the vocabulary of $\phi$, $P$ is the domain of $m$ and $R$ is the single binary relation of $m$).
\item $\theta'(E,u)$ states $E$ is well-founded and extensional (so that $E$ has a transitive collapse, and the domain of the model equipped with $E$ can be thought of as a transitive set), and its collapse is a level of $L[u]$ having a normal measure $u$ as an element.
\item $\phi'(P,R')$ is (as before) a modification of the sentence $\phi$ where each first order quantifier $\exists x \dots$ is relativised to $P$ as $\exists x(P(x) \land \dots)$, and each occurrence of $R$ is replaced by $R'$.
\item $\theta_{least}(E,u,m)$ says $m <_{L[u]} m'$ for any other $m' = (Q,S)$ also satisfying $\phi'(Q,S)$ (using the formula defining the canonical well-order of $L[u]$ with $u$ as a parameter).
\item $\theta_{isom}(E,m,P,R')$ states that $m = (P,R')$, and that $(P,R')$ is isomorphic to the ambient model (so there is an injection $F$ with range $P$ such that $R(x,y) \leftrightarrow R'(F(x),F(y))$ for all $x,y$).
\end{itemize}
If $M \models \theta$ with witnesses $E$, $u$ and $m = (P,R')$, and $\pi \colon (M,E) \to (N,\in)$ is the transitive collapse, then $\pi(u) = U$ is the unique normal measure $U$ on $\kappa$, $N = L_\alpha[U]$ for some $\alpha$ and $\pi(m)$ is the $<_{L[U]}$-least model $M_0$ of $\phi$, so $M$ is isomorphic to $M_0$.

Conversely, let $\alpha$ be least such that $M_0 \in L_\alpha[U]$. Then $\kappa < \alpha < \abs{M_0}^+$ and $U \in L_\alpha[U]$, so we may pick a bijection $\pi \colon M_0 \to L_\alpha[U]$ and let $E$, $u$ and $m = (P,R')$ be the preimages of $\in$, $U$ and $M_0$ under $\pi$ to witness $M_0 \models \theta$.

Thus the above sentence $\theta$ is such that $M \models \theta$ if and only if $M$ is isomorphic to the $<_{L[U]}$-least model of $\phi$. Now if $M \models \phi$, also $M \models \theta$ by completeness of $\phi$, so $M$ is isomorphic to $M_0$ and $\phi$ is categorical.

\medskip

Suppose now that $\abs{M_0} = \lambda < \kappa$. In this case we cannot find a binary relation $E$ on $M_0$ and $u \in M_0$ such that $u$ is a normal measure in the transitive collapse of $(M_0,E)$, so we modify the previously produced sentence $\theta$. This argument relies on a straightforward modification of the $\Delta^1_3$ well-order of reals in $L[U]$. We make the further assumption that the domain of $M_0$ is a cardinal (and that $M_0$ is the $<_{L[U]}$-least among such models), and let $\theta$ be the sentence
\[
\exists E \exists W \exists m \exists P \exists R'(\theta'(E,W) \land \phi'(E,P,R') \land \theta_{least}(E,W,m) \land \theta_{isom}(E,m,P,R'))
\]
where
\begin{itemize}
\item $E, R'$ are binary and $W,P$ unary predicate symbols, and $m$ a first order variable, none occurring in $\phi$.
\item $\theta'(E,W)$ states $E$ is well-founded and extensional, whose transitive collapse is an iterable $L[U]$-premouse $(L_\alpha[W],\in,W)$ for some $\alpha$, where $W$ is a $L[W]$-ultrafilter on some $\gamma$, where $\gamma$ is an ordinal greater than the cardinality of the ambient model.
\item $\phi'(E,P,R')$ is the sentence $\phi'(P,R')$ from before, with the additional stipulation that the extent of the predicate $P$ is a cardinal.
\item $\theta_{least}(E,W,m)$ says $m <_{L[W]} m'$ for any other $m' = (Q,S)$ also satisfying $\phi'(E,Q,S)$ (using the formula defining the canonical well-order of $L[W]$ with $W$ as a predicate).
\item $\theta_{isom}(E,m,P,R')$ remains unchanged from earlier.
\end{itemize}
We claim that $\theta$ is a sentence such that $M \models \theta$ if and only if $M$ is isomorphic to the $<_{L[U]}$-least model of $\phi$ (among models whose domain is a cardinal). So suppose $M \models \theta$ with witnesses $E, W$ and $m = (P,R')$, and let $\pi \colon (M,E,W) \to (N,\in,W')$ be the transitive collapse. Then $W' = \pi''(W)$ is a $N$-ultrafilter on some $\gamma > \lambda$ and $N = L_\alpha[W']$ for some $\alpha > \gamma$, and $\pi(m)$ is the $<_{L[W']}$-least model of $\phi$ in $L_\alpha[W']$, to which $M$ is isomorphic.

To see why $\pi(m)$ is $M_0$, let $j \colon L[U] \to L[F]$ and $k \colon L_\alpha[W'] \to L_\delta[F]$ be long enough iterations of $L[U]$ and $L_\alpha[W]$ respectively such that they become comparable. Then $\crit(j) = \kappa > \lambda$ and $\crit(k) = \gamma > \lambda$, so $j(M_0) = M_0$ and $k(\pi(m)) = \pi(m)$. By elementarity, both $M_0$ and $\pi(m)$ are now the $<_{L[F]}$-least model of $\phi$ among models whose domain is a cardinal, so $\pi(m) = M_0$ and $M$ is isomorphic to $M_0$.

Conversely, to see $M_0 \models \theta$ amounts to finding an appropriate premouse $(L_\alpha[W], \in, W)$. Let $\delta$ be a large enough cardinal such that $M_0, U \in L_\delta[U]$, and that $(L_\delta[U],\in,U)$ is an iterable premouse. Then let $N$ be the Skolem hull of $\lambda \cup \set{M_0}$ in $L_\delta[U]$ of cardinality $\lambda$, and let $\pi \colon (N,\in,U \cap N) \to (L_\alpha[W], \in, W)$ be the transitive collapse. Now $(L_\alpha[W],\in,W)$ is also an iterable premouse with $\abs{L_\alpha[W]} = \lambda$, $W$ is a $L_\alpha[W]$-ultrafilter on some $\gamma = \pi(\kappa) > \lambda$, and $\pi(M_0) = M_0$, so by elementarity $M_0$ is the $<_{L[W]}$-least model of $\phi$ as required. So $\theta$ is a sentence such that $M \models \theta$ if and only if $M$ is isomorphic to $M_0$, implying as before that $\phi$ is categorical.

Finally, observe that the case $\abs{M_0} = \kappa$ is impossible, since the measurable cardinal $\kappa$ is $\Pi^2_1$-indescribable \cite{hanfscott}. Thus $\phi$ is categorical.
\end{proof}

It turns out that finite axiomatizability is key for the preceding theorem. For every second order characterizable cardinal $\lambda > \kappa$, we produce a non-categorical recursively axiomatizable theory whose models have cardinality $\lambda$.

\begin{theorem} \label{ultrafilter-noncat}
Assume $V = L[U]$. Suppose $\kappa$ is measurable and $\lambda$ is second order characterizable with $\lambda > \kappa$. Then there is a recursively axiomatizable theory $T$ with $\kappa$ many non-isomorphic models of cardinality $\lambda$.
\end{theorem}
\begin{proof}
For $\alpha < \kappa$ let $M_\alpha = (\lambda + \alpha, <)$, so in a structure of cardinality $\lambda$, $M_\alpha$ is straightforwardly definable from $\alpha$ (as $\lambda$ is second order characterizable). These models have the property that $M_\alpha \isom M_\beta$ implies $M_\alpha = M_\beta$.

For a second order sentence $\phi$ in vocabulary $(<)$, let
\[
S_\phi = \set{\alpha < \kappa : M_\alpha \models \phi},
\]
and let $T_0$ be the set of sentences $\phi$ such that $S_\phi \in U$. As $U$ is an ultrafilter, $T_0$ is a complete theory (so for any $\phi$, exactly one of $\phi \in T_0$ or $\lnot \phi \in T_0$ hold), and by the $\sigma$-completeness of $U$ the intersection $X = \bigcap \set{S_\phi : \phi \in T_0} \in U$ is nonempty. The set $X$ is such that for any $\alpha, \beta \in X$, the structures $M_\alpha$, $M_\beta$ have the same second order theory $T_0$, so it remains to see that the theory $T_0$ is recursively axiomatizable.

For a second order sentence $\phi$ in vocabulary $(<)$, let $E$ be a binary relation symbol and $u$ a first order variable, neither occurring in $\phi$, and let $\phi^+$ be the second order sentence 
\[
\exists E \exists u (\theta'(E,u) \land (\exists x \in u)(\forall \alpha \in x)"M_\alpha \models \phi")
\]
where $\theta'(E,u)$ says $E$ is well-founded and extensional, and its transitive collapse is a level of $L[u]$ containing $\lambda$ and having a normal measure $u$ as an element. Note that $\phi^+$ is a sentence in the empty vocabulary. Intuitively, $\phi^+$ states that $M_\alpha \models \phi$ for a $U$-big set of ordinals $\alpha < \kappa$, so for any structure $N$ with $\abs N = \lambda$ we have the equivalences
\begin{align*}
N \models \phi^+ &\iff \set{\alpha < \kappa : M_\alpha \models \phi} = S_\phi \in U \\
&\iff M_\alpha \models \phi \text{ for some } \alpha \in X \\
&\iff \phi \in T_0.
\end{align*}
The import of the vocabulary of $\phi^+$ being empty is that for a structure $N$, the truth of $N \models \phi^+$ depends only on $\abs N$, so we get that for all structures $N$ with $\abs N = \lambda$,
\[
N \models \phi^+ \iff M_\alpha \models \phi^+ \text{ for some } \alpha \in X \iff \phi^+ \in T_0
\]
so also $\phi \lequ \phi^+ \in T_0$ for all second order sentences $\phi$ in vocabulary $(<)$.

Now define the recursive set of sentences
\[
T = \set{\phi \lequ \phi^+ : \phi \text{ is a second order sentence in vocabulary } (<)}.
\]
Observe that any model $N$ of the theory $T$ has cardinality $\lambda$, since taking $\theta_\lambda$ to be the second order characterization of $\lambda$, we have $M_\alpha \models \theta_\lambda$ for all $\alpha < \kappa$, so $N \models \theta_\lambda^+$ and thus $N \models \theta_\lambda$ since $\theta_\lambda^+ \lequ \theta_\lambda \in T$.

To see $T$ axiomatizes $T_0$, suppose $N \models T$ so $\abs N = \lambda$, and that $\phi$ is a second order sentence in the vocabulary $(<)$, so either $\phi \in T_0$ or $\lnot \phi \in T_0$. In the former case we have $S_\phi \in U$ so $N \models \phi^+$, so $N \models \phi$, and in the latter case we have $S_{\lnot \phi} \in U$ so $N \models \lnot \phi$. Thus $\Th_2(N) = T_0$, so $T$ recursively axiomatizes $T_0$ as desired.
\end{proof}

In conclusion, all complete finitely axiomatizable theories are categorical in $L[U]$ as in $L$, and in $L[U]$ there are complete recursively axiomatizable second order theories that are non-categorical (whereas this is still unknown in $L$).

\section{Countable models}

We already remarked earlier that if $V=L$, then every complete finitely axiomatized 
second order theory is categorical \cite{Solo}.
We now show that for theories with a countable model this is a consequence of PD, and therefore a consequence of large cardinals:

\begin{theorem}\label{one}
Assume PD. Every complete finitely axiomatized 
second order theory with a countable model is categorical.
\end{theorem}

\begin{proof}
Suppose $\phi$ is a complete second order sentence with a countable model. Then by completeness all models of $\phi$ are countable. Suppose $\phi$ is on the level $\Sigma^1_n$ of second order logic and its vocabulary is, for simplicity, just one binary predicate symbol $P$. Let $R$ be the $\Sigma^1_n$ (lightface) set of real numbers coding  models of $\phi$. By PD and its consequence, the Projective Uniformization Theorem \cite[Theorem 6C5]{MR2526093}, there is a $\Sigma^1_{n+1}$ (even $\Sigma^1_n$ if $n$ is even) element $r$  in $R$. Suppose $r$ codes the model $M$ of $\phi$.  We show that every model of $\phi$ is isomorphic to $M$. Suppose $N$ is a model of $\phi$.  Let $\theta$ be the following second order sentence:

$$\begin{array}{l}
\exists Q_+\exists Q_{\times}(\theta_1(Q_+,Q_{\times})\wedge\exists A
(\theta_2(Q_+,Q_{\times},A)\wedge\\
\exists B(\theta_3(Q_+,Q_{\times},A,B)\wedge
\exists F\theta_4(F,B)))),
\end{array}$$

\noindent where \begin{itemize}
\item $\theta_1(Q_+,Q_{\times})$ is the standard second order characterization of $(\N,+,\times)$.
\item $\theta_2(Q_+,Q_{\times},A)$ says that the set $A$ satisfies the $\Sigma^1_{n+1}$ definition of $r$ in terms of $Q_+$ and $Q_{\times}$. 
\item $\theta_3(Q_+,Q_{\times},A,B)$ says in a domain $N$ that $(N,B)$ is the binary structure coded by $A$ in terms of $Q_+$ and $Q_{\times}$.
\item $\theta_4(F,B)$ is the second order sentence which says that $F$ is a bijection and $$\forall x\forall y(P(x,y)\leftrightarrow B(F(x),F(y))).$$
\end{itemize}
Thus, $\theta$ essentially says ``I am isomorphic to the model coded by  $r$." Trivially, $M\models\theta$. Recall that $M\models\phi$. Since $\phi$ is complete, $\phi\models\theta$. Therefore our assumption  $N\models\phi$ implies $N\models\theta$ and therefore $N\cong M$. 
\end{proof}

We make a few remarks about the proof. First, if $n=2$, then we can use the Novikov-Kondo-Addison Uniformisation Theorem and PD is not needed. Thus we obtain:

\begin{corollary}
A complete $\Sigma^1_2$-sentence of second order logic with a countable model is always categorical.
\end{corollary}

In fact, the  categorical finite second order axiomatizations of structures such as $(\N,+,\times)$, $(\R,+,\times,0,1)$ and $(\C,+,\times, 0,1)$ (any many other classical structures) are all on the $\Pi^1_1$-level of second order logic.

Second, the above proof gives also the following more general result: Assume $Det(\bDelta^1_{2n})$. Suppose $T$ is a recursively axiomatized  theory on the $\Sigma^1_{2n+2}$-level of second order logic, which is complete for sentences on this level of second order logic. Then $T$ is categorical.

An essential ingredient of the proof of Theorem~\ref{one} was the assumption that the complete second order theory is finitely axiomatized. The following theorem shows that ``finitely" cannot be replaced by ``recursively".

\begin{theorem}\label{ones}
Assume PD. There is a recursively axiomatized  complete 
second order theory with $2^\omega$ non-isomorphic countable models.
\end{theorem}

\begin{proof}
For any $x\subseteq\omega$ let  

$$M_x=(V_\omega\cup\{y\subseteq\omega : y\equiv_T x\},\in),$$ 

\noindent where $y\equiv_Tx$ means the Turing-equivalence of $y$ and $x$. We denote the second order theory of $M_x$ by $\Th_2(M_x)$. By construction, $x\equiv_Ty$ implies $\Th_2(M_x)=\Th_2(M_y)$. On the other hand, if $x\not\equiv_T y$, then clearly $M_x\ncong M_y$. If $\phi$ is a second order sentence, then `$M_x\models\phi$' is a projective property of $x$, closed under $\equiv_T$, and hence by Turing Determinacy for projective sets  \cite{MR0227022} has a constant truth value on a cone of reals $x$. By intersecting the cones we get a cone $C$ of reals $x$ on which $\Th_2(M_x)$ is constant. For any second order $\phi$ let $\phi^+$ be the second order sentence

$$``M_y\models\phi\mbox{ for a cone of $y$}"$$

\noindent and $\hat{\phi}$ the sentence $\phi\leftrightarrow\phi^+$. Let us consider the recursive second order theory $T$ consisting of $\hat{\phi}$, when $\phi$ ranges over second order sentences in the vocabulary of the structures $M_x$. We may immediately conclude that $T$ is complete, for if a second order sentence $\phi$ is given, then by the choice of $C$  either $M_x\models\phi$ for  $x\in C$ or $M_x\models\neg\phi$ for  $x\in C$. In the first case $\hat{\phi}\models\phi$ and in the second case $\hat{\phi}\models\neg\phi$. Therefore, $T\models\phi$ or $T\models\neg\phi$. There are a continuum of non Turing equivalent reals in the cone $C$. Hence there are a continuum of non-isomorphic $M_x$ with $x\in C$.
\end{proof}

\section{Models of cardinality $\aleph_1$}

Next, we show that the $(*)$ axiom (see Definition 4.33 in \cite{MR2723878})
has categoricity consequences for theories with a model of cardinality $\aleph_1$. Thus these consequences can also be derived from forcing axioms, namely MM$^{++}$ which implies the $(*)$ axiom (as shown in \cite{MMplusplusstar}).
The following theorem answers a question of Boban Veli\v{c}kovi\'{c}.
\begin{theorem}
Assume $(*)$. Then there is a complete finitely axiomatizable second order theory with $\omega_2 \,(=2^{\omega_1})$ non-isomorphic models of cardinality $\aleph_1$.
\end{theorem}
\begin{proof}
The pertinent consequence of $(*)$ is the quasihomogeneity of the nonstationary ideal on $\omega_1$ (see Section 5.8 in \cite{MR2723878}, particularly Definition 5.100). We take ``NS$_{\omega_1}$ is quasihomogeneous'' to be the following statement: if $X \subseteq \powset(\omega_1)$ is ordinal definable from parameters in $\R \cup \{$NS$_{\omega_1}\}$, and $X$ is closed under equality modulo NS$_{\omega_1}$, and $X$ contains one bistationary subset of $\omega_1$, then $X$ contains every bistationary subset of $\omega_1$.

We focus on the $\omega_1$-like dense linear orders $\Phi(S) = \eta + \sum_{\alpha < \omega_1} \eta_\alpha$, where
\[
\eta_\alpha = \begin{cases} \eta, & \alpha \notin S \\ 1+\eta, & \alpha \in S, \end{cases}
\]
$\eta$ is the order type of the rationals, and $S \subseteq \omega_1$ is bistationary. These models have the property that $\Phi(S) \isom \Phi(S')$ if and only if $S \triangle S' \in NS_{\omega_1}$. For a second order sentence $\phi$  in vocabulary $(<)$, the set
\[
X_\phi = \set{S \subseteq \omega_1 : S \text{ bistationary}, \Phi(S) \models \phi}
\]
is ordinal definable, and closed under equality modulo NS$_{\omega_1}$, so the quasihomogeneity of NS$_{\omega_1}$ implies that $X_\phi$ contains either every bistationary subset of $\omega_1$, or none of them.

This shows the models $\Phi(S)$ for bistationary $S \subseteq \omega_1$ all have the same complete second order theory, which is thus non-categorical. This theory is axiomatized by the second order sentence in vocabulary $(<)$ expressing ``I am isomorphic to $\Phi(S)$ for some bistationary $S \subseteq \omega_1$'', so it is finitely axiomatizable, as required.
\end{proof}

Some categoricity consequences of $(*)$ can already be derived from AD, the axiom of determinacy. As the axiom $(*)$ states that $L[\powset(\omega_1)]$ is a homogeneous forcing extension of a model of AD by a forcing that does not add reals, the categoricity consequences of AD for theories with a model of cardinality $\leq \aleph_1$ also follow from $(*)$. (Of course, the existence of recursively axiomatized non-categorical theories under $(*)$ is overshadowed by the existence of even finitely axiomatized such theories.)


\begin{theorem}
Assume AD. Then there is a complete recursively axiomatized second order theory with at least $2^{\aleph_0}$ many models of cardinality $\aleph_1$.
\end{theorem}
\begin{proof}
By Martin, AD implies  $\omega_1 \to (\omega_1)^\omega$, and moreover the homogeneous set given by $\omega_1 \to (\omega_1)^\omega$ can be taken to be a club (see \cite{MR0479903}). We may then intersect $\omega$ many homogeneous clubs for $\omega$ many colorings to obtain  $\omega_1 \to (\omega_1)^\omega_{2^\omega}$, and the homogeneous subset can still be taken to be a club.

We focus on models of the form $M_X = (\omega_1, <, X)$ for $X \in [\omega_1]^\omega$. The second order theory $\Th_2(M_X)$ in the vocabulary $(<,X)$ can be encoded by a real $f(X) \in 2^\omega$ consisting of the G\"{o}del numbers of the sentences true in $M_X$. This gives a coloring $f \colon [\omega_1]^\omega \to 2^\omega$, so we find a homogeneous club subset $H_0 \subseteq \omega_1$ such that $f(X)$ does not depend on $X \in [H_0]^\omega$. Hence the models $M_X$ with $X \in [H_0]^\omega$ all have the same complete second order theory $T_0$, which is thus non-categorical.

The theory $T_0$ is axiomatized by
\[
T = \set{\phi \lequ \phi^+ : \phi \text{ is a second order sentence}}
\]
where for a given second order sentence $\phi$ in vocabulary $(<,X)$, the sentence $\phi^+$ expresses ``there exists a club $C \subseteq \omega_1$ such that $M_X \models \phi$ for all $X \in [C]^\omega$''.

For a given second order sentence $\phi$, if $M_X \models \phi$ for each $X \in [H_0]^\omega$, then $H_0$ serves to witness that $\phi^+$ holds, so $T \models \phi$. Conversely, if $\phi^+$ holds, there is a club $C$ such that $M_X \models \phi$ for every $X \in [C]^\omega$, and taking $X \in [C \cap H_0]^\omega$ we see also that $M_X \models \phi$ for all $X \in [H_0]^\omega$. Thus $T \models \phi$ for exactly those $\phi$ such that $M_X \models \phi$ for all $X \in [H_0]^\omega$, so we see that $T$ is a recursive axiomatization of the theory $T_0$ as desired.
\end{proof}

The same can be analogously derived from the $(*)$ axiom, as follows:

\begin{corollary}
Assume $(*)$. Then there is a complete recursively axiomatized second order theory with $\omega_2$ many models of cardinality $\aleph_1$.
\end{corollary}
\begin{proof}
Recall $(*)$ states that $L[\powset(\omega_1)] = L(\R)^{\Pmax}$ and AD holds in $L(\R)$. As $\Pmax$ is homogeneous and does not add reals under AD (see Lemmas 4.40 and 4.43 in \cite{MR2723878}), $\omega_1 = \omega_1^{L(\R)}$ and $[\omega_1]^\omega = ([\omega_1]^\omega)^{L(\R)}$.

We again look at models $M_X = (\omega_1, <, X)$ for $X \in [\omega_1]^\omega$, and working in $L(\R)$, define a coloring $f \colon [\omega_1]^\omega \to 2^\omega$ by
\[
f(X) = r \quad \iff \quad L(\R) \models \Pmax \force "\check r \text{ codes } \Th_2(M_{\check X})".
\]
That $f$ is a well-defined total function relies on the homogeneity of $\Pmax$.
By AD$^{L(\R)}$ we find a club $H_0 \in L(\R)$, $H_0 \subseteq \omega_1$ homogeneous for $f$.
Stepping out of $L(\R)$, we see that the models $M_X$, $X \in [H_0]^\omega$ all have the same complete second order theory $T_0$ (in $L(\R)^\Pmax = L[\powset(\omega_1)]$ and in $V$ both).

Working now in $V$, we again define
\[
T = \set{\phi \lequ \phi^+ : \phi \text{ is a second order sentence}}
\]
where for a given second order sentence $\phi$, the sentence $\phi^+$ expresses ``there exists a club $C \subseteq \omega_1$
such that $M_X \models \phi$ for all $X \in [C]^{\omega}$''.
The proof concludes analogously to the preceding theorem.

We note that $(*)$ calculates $\abs{\omega_1^\omega}$ to be $\omega_2$, so $T_0$ has $\omega_2$ many non-isomorphic models as claimed.
\end{proof}

Of course, we may also use the fact that the club filter on $\omega_1$ is an ultrafilter under AD to get another complete recursively axiomatized non-categorical second order theory, the difference being that this theory has $\omega_1$ many models instead. The proof, analogous to the proof of Theorem \ref{ultrafilter-noncat}, is omitted:

\begin{theorem}
Assume AD. Then there is a complete recursively axiomatized second order theory with $\omega_1$ many models of cardinality $\aleph_1$. \qed
\end{theorem}

This proof is also easily modified to assume $(*)$ instead:
\begin{corollary}
Assume $(*)$. Then there is a complete recursively axiomatized second order theory with $\omega_1$ many models of cardinality $\aleph_1$. \qed
\end{corollary}

Thus, under $(*)$, a complete non-categorical theory with a model of cardinality $\aleph_1$ may have either $\omega_1$ or $\omega_2$ many non-isomorphic models.

\section{Forcing non-categoricity}

We shall show (Theorem~\ref{aleph1}) that we can force, over any model of set theory, a finite complete non-categorical second order theory with a model of cardinality $\aleph_1$. This shows that large cardinals cannot imply the categoricity of finite complete second order theories in general and, in particular, in the case that the theory has a model of cardinality $\aleph_1$. This is in  contrast to finite complete second order theories with a countable model where PD implies  categoricity (Theorem~\ref{one}).

Here is an outline of the proof. We start with a preparatory countably closed forcing $\bP$ obtaining a generic extension $V[G]$. Then we add $\aleph_1$ Cohen-reals obtaining a further generic extension $V[G][H]$. In this model we consider for every $x\subseteq\omega$ the model
\begin{equation}\label{HC}
M_x=(HC^{V[x]},HC^V,\in).
\end{equation}
We  show that if $x$ is Cohen-generic over $V[G]$, then the complete second order theory of $M_x$ is  finitely axiomatizable (in second order logic), and if $x$ and $y$ are mutually Cohen-generic over $V[G]$, then $M_x$ and $M_y$ are second order equivalent but non-isomorphic.

We begin by recalling the following \emph{fast club} forcing $\bP_\fc$, due to R. Jensen:
%
Conditions are pairs $p=(c_p,E_p)$ where $c_p$ is a countable closed subset of $\omega_1$ and  $C_p$ is club in $\omega_1$. We define $(c_p,E_p) \le (c_q,E_q)$ if $c_q$ is an initial segment of $c_p$,  $E_q \subseteq E_p$, and $c_p \setminus c_q \subseteq E_q$.
%
This forcing is countably closed. If we assume CH, this forcing has the $\aleph_2$-c.c. It is called {\it fast club forcing} because of the following property: Suppose $G$ is $\bP_{\fc}$-generic. If $C_G$ is the union of the sets $c_p$ such that $p\in G$, then the following holds: If $D$ is any club in $V$, then there is $\alpha$ such that $C_G\setminus\alpha\subseteq D$.
The set  $C_G$ is called a \emph{fast club} (over $V$).

Let $\bQ$ be the poset $\fn(\omega_1\times\omega,2,\omega)$ for adding $\aleph_1$ Cohen reals. We use fast club forcing to build a preparatory iterated forcing in such a way that after forcing with $\bQ$ the ground model reals are second order definable from any set $A\subseteq\omega_1$ with a certain second order property. The following lemma is crucial in the iteration:

\begin{lemma}\label{code357}Suppose $G\times H$ is $\bP_\fc\times\bQ$-generic over $V$. Suppose  $A\subseteq\omega_1$ is in 
$V[H]$ and  $D\subseteq C_G$ is a club in $V[G\times H]$ such that $V[G\times H]$ satisfies $\forall \alpha<\omega_1(D\cap\alpha\in L[A])$.
Then $\P(\omega)^V\subseteq L[A]$.
\end{lemma}

\begin{proof} We modify a construction from the proof of \cite[Lemma 4.33]{MR3632568} to our context.
Let us call a pair $(A,B)$ of sets of ordinals  an {\emph{interlace}}, if
 $A\cap B=\emptyset$, above every element of $A$ there is an element of $B$, and vice versa.
 Suppose we have disjoint sets $X, Y, Z\subseteq \omega_1$ such that $(X\cup Y,Z)$ is an interlace.
 Let $z\sim z'$ in $Z$ if $(z,z')\cap(X\cup Y)=\emptyset$.
 Let $[z_n]$, $n<\omega$, be the first $\omega$ $\sim$-equivalence classes in $Z$ in increasing order.
The triple $(X,Y,Z)$ is said to {\emph{code}} the set $a\subseteq\omega$ if for all $n<\omega$:

 $$n\in a\iff
\min\{\alpha\in X\cup Y:[z_n]<\alpha<[z_{n+1}]\}\in X.$$

It suffices to prove that for every $a\subseteq\omega$ in $V$ there is a triple $(X,Y,Z)\in L[A]$ such that $(X\cup Y,Z)$ is an interlace, and $(X,Y,Z)$ codes $a$. To this end, suppose $a\in\P(\omega)^V$.

Suppose $\dot{A}$ is a $\bQ$-name for $A$ in $V$, $\tau\in V$ is a $\bP_\fc$-name for a $\bQ$-name
  $\dot{D}$  for  $D$, and $\dot{F}$ a $\bQ$-name for a function $\omega_1\to\omega_1$ which lists the elements of $\dot{D}$ in increasing order. W.l.o.g. $\tau$ is a $\bP_\fc$-name $\langle \dot{f}_\alpha:\alpha<\omega_1\rangle$ for a sequence  of countable partial functions defined on $\omega_1$ such that $\{\dot{f}_\alpha(\gamma):\gamma\in \dom(f_\alpha)\}$ is a maximal antichain in $\bQ$ and $\dot{f}_\alpha(\gamma)$ forces $\dot{F}(\alpha)=\gamma$. 
  Suppose  (w.l.o.g.) the weakest condition in $\bP_\fc\times\bQ$ forces  what is assumed about $\dot{A}$, $\dot{F}$, $\tau$ and $\dot{D}$. Since $\bP_\fc\force``\bQ\force \dot{D}\subseteq C_{\dot{G}}"$, we have $\force\dom(\dot{f}_\alpha)\subseteq C_{\dot{G}}$. More generally, if $p\in\bP_\fc$ decides the countable set $\dom(\dot{f}_\alpha)$, then 
  \begin{equation}\label{included}
p\force \dom(\dot{f}_\alpha)\subseteq c_p\setminus\alpha.
\end{equation}

If $\delta<\omega_2$, let $W_\delta$ be the set of conditions $p\in\bP_\fc$ such that $p$ decides $\dom(\dot{f}_\delta)$. It is easy to see that $W_\delta$ is dense.

We construct descending $\omega$-sequences $(p_n), (q_n)$ and $(r_n)$ in $\bP_\fc$ as follows.
We let $p_0=q_0=r_0$ be the weakest condition in $\bP_\fc$. Suppose $p_n, q_n$ and $r_n$ have been defined already. Let $\delta_n=\max(c_{r_n}\cup{\{0\}})$. Now there are two cases:
\begin{enumerate}
\item Case $n\in a$:
\begin{enumerate}
\item Let $p_{n+1}\le p_n$ such that $\min(c_{p_{n+1}}\setminus c_{p_n})>\delta_n$ and $p_{n+1}\in W_{\delta_n}$. 
 \item Let $q_{n+1}\le q_n$ such that $\min(c_{q_{n+1}}\setminus c_{q_n})>\max(c_{p_{n+1}})$  and $q_{n+1}\in W_{\delta_n}$.
\item Let $r_{n+1}\le r_n$ such that $\min(c_{r_{n+1}}\setminus c_{r_n})>\max(c_{q_{n+1}})$ and $q_{n+1}\in W_{\delta_n}$.
\end{enumerate}
\item Case $n\notin a$:  
\begin{enumerate}
\item Let $q_{n+1}\le q_n$ such that $\min(c_{q_{n+1}}\setminus c_{q_n})>\delta_n$ and $q_{n+1}\in W_{\delta_n}$.
 \item Let $p_{n+1}\le p_n$ such that $\min(c_{p_{n+1}}\setminus c_{p_n})>\max(c_{q_{n+1}})$  and $p_{n+1}\in W_{\delta_n}$.
\item Let $r_{n+1}\le r_n$ such that $\min(c_{r_{n+1}}\setminus c_{r_n})>\max(c_{p_{n+1}})$ and $r_{n+1}\in W_{\delta_n}$.
\end{enumerate}
\end{enumerate}

Note that if $\delta_n<\alpha<\min(c_{p_{n+1}}\setminus c_{p_n})$, then $p_{n+1}\force \alpha\notin C_{\dot{G}}$, whence $p_{n+1}\force \alpha\notin\tau$. Respectively, if $\delta_n<\alpha<\min(c_{q_{n+1}}\setminus c_{q_n})$, then $q_{n+1}\force \alpha\notin C_{\dot{G}}$, whence $q_{n+1}\force \alpha\notin\tau$, and if $\delta_n<\alpha<\min(c_{r_{n+1}}\setminus c_{r_n})$, then $r_{n+1}\force \alpha\notin C_{\dot{G}}$, whence $r_{n+1}\force \alpha\notin\tau$. 

Similarly, if $\max(c_{p_{n+1}})<\alpha<\delta_{n+1}$, then $p_{n+2}\force \alpha\notin C_{\dot{G}}$, whence $p_{n+2}\force \alpha\notin\tau$. Respectively, if $\max(c_{q_{n+1}})<\alpha<\delta_{n+1}$, then $q_{n+2}\force \alpha\notin C_{\dot{G}}$, whence $q_{n+2}\force \alpha\notin\tau$.

Finally, if $\alpha\in I=[\min(c_{p_{n+1}}),\max(c_{p_{n+1}})]$, then $p_{n+1}$ may leave the sentence $\alpha\in\tau$ undecided, but still $p_{n+1}\force I\cap\tau\ne\emptyset$, since $p_{n+1}$ decides $\dom(\dot{f_{\delta_n}})$ and we have (\ref{included}). Respectively,  $q_{n+1}$ forces $[\min(c_{q_{n+1}}),\max(c_{q_{n+1}})]\cap\tau\ne\emptyset$, and $r_{n+1}$ forces $[\min(c_{r_{n+1}}),\max(c_{r_{n+1}})]\cap\tau\ne\emptyset$.

 Let $p_\omega=\inf_np_n, q_\omega=\inf_nq_n, r_\omega=\inf_nr_n$, and
 let $\delta=\sup\{\delta_n: n<\omega\}$.
 Let $G_0\subseteq\bP_\fc$ be generic over $V[H]$ such that $p_\omega\in G_0$, 
 $G_1\subseteq\bP_\fc$ generic over $V[H]$ such that $q_\omega\in G_1$, and 
  $G_2\subseteq\bP_\fc$  generic over $V[H]$ such that $r_\omega\in G_2$. 
 Lastly, let $$
X=\tau_{G_0\times H}\cap\delta, \ Y=\tau_{G_1\times H}\cap\delta, \  Z=\tau_{G_2\times H}\cap\delta.$$ 
As $\force_{\bP_\fc\times\bQ}\tau\cap\delta\in L[\dot{A}]$ and $\dot{A}_{G_0\times H}=\dot{A}_H$, we have $V[G_0\times H]\models X\in L[A]$. By absoluteness,  $V[H]\models X\in L[A]$. Similarly, $V[H]\models Y,Z\in L[A]$. 
 %


Now by construction, $(X\cup Y,Z)$ is an interlace and $(X,Y,Z)$ codes $a$. Hence $a\in L[A]$.
\medskip

\end{proof}

We need another auxiliary lemma for the iteration:

\begin{lemma}\label{closed}
 Assume 
 $G$ is $\bP_\fc$-generic over $V$,  $\bR\in V[G]$ is a $\sigma$-closed forcing,
 $K$ is $\bR$-generic over $V[G]$,
 $H$ is $\bQ$-generic over $V[G][K]$, 
 $A\subseteq\omega_1$ is in $V[H]$, and
 in $V[G][K][H]$, there
is a club $D \subseteq C_{G}$ such that $D\cap\alpha \in L[A]$ for all $\alpha < \omega_1$.
Then such a club $D$ must already exist in $V[G][H]$.
\end{lemma}

\begin{proof}
Suppose $\dot{A}\in V$ is a $\bQ$-name for $A$ and $\dot{D}\in V[G]$ is an $\bR$-name for a $\bQ$-name  
   for  $D$. Suppose  $\dot{F}\in V[G]$  is an $\bR$-name for a $\bQ$-name   for a function $\omega_1\to\omega_1$ listing the elements of $\dot{D}$ in increasing order. W.l.o.g. $\dot{D}$ is a $\bR$-name $\langle \dot{f}_\alpha:\alpha<\omega_1\rangle$ for a sequence  of countable partial functions defined on $\omega_1$ such that $\{\dot{f}_\alpha(\gamma):\gamma\in \dom(f_\alpha)\}$ is a maximal antichain in $\bQ$ and $\dot{f}_\alpha(\gamma)$ forces $\dot{F}(\alpha)=\gamma$. 
  Suppose  (w.l.o.g.) the weakest condition in $\bR\times\bQ$ forces  what is assumed about $\dot{A}$, $\dot{F}$, and $\dot{D}$. Since $\force \dot{D}\subseteq C_{\dot{G}}$, we have $\force\dom(\dot{f}_\alpha)\subseteq C_{\dot{G}}$. 

We shall define a descending sequence $(r_\alpha)_{\alpha<\omega_1}$ in $K$. For a start, $r_0\in K$ can be arbitrary. Suppose $r_\alpha\in K$ has been defined already. 
Let $r_{\alpha+1}\le r_\alpha$ such that $r_{\alpha+1}\in K$ and $r_{\alpha+1}$ decides
 $\dom(\dot{f}_{\beta})$ and $\dot{f}_\alpha(\gamma)$ for $\beta\le\alpha$ and $\gamma\in \dom(\dot{f}_\alpha)$. 
Let $g_\alpha\in V[G]$ such that $r_{\alpha+1}\force \dot{f}_\alpha=g_\alpha$. Let $\dot{S}$ be a $\bQ$-name in $V$ for a function $\omega_1\to\omega_1$
such that $g_\alpha(\gamma)\force\dot{S}(\alpha)=\gamma$. Let $\dot{E}\in V$ be a $\bQ$-name such that $\force\dot{E}=\{\dot{S}(\alpha):\alpha<\omega_1\}$.
Now 

$$V[K][H]\models\dot{E}_H=\dot{D}_{K\times H}\ \wedge\ \dot{E}_H\cap\delta\in L[A],$$

\noindent whence $V[H]\models \dot{E}_H\cap\delta\in L[A]$ follows by absoluteness.
\end{proof}

Now we can construct the iteration in such a way that after forcing with the iteration and then with $\bQ$ the ground model reals, which are the same as the reals after the iteration,  are second order definable from any set $A\subseteq\omega_1$ with a certain second order property.

\begin{lemma}\label{eww} We assume CH.
Suppose $\bP$ is the countable support iteration of fast club forcing of length $\omega_2$. Let $G$ be $\bP$-generic over $V$.
Suppose $H$ is $\bQ$-generic over $V[G]$. Suppose 
in
$V[G][H]$ there is a set $A\subseteq \omega_1$ such that  for every club $C$, there is a club $D \subseteq C$ such that $D\cap \alpha \in L[A]$ for all $\alpha < \omega_1$. Then $P(\omega)^V \subseteq L[A]$.
\end{lemma}

\begin{proof}  Let $\bP=\langle \bP_\alpha:\alpha<\omega_2\rangle$ be the countable support  iteration of $\langle \dot{Q}_\alpha:\alpha<\omega_2\rangle$, where $\bP_\alpha\force ``\dot{Q}_\alpha\mbox{ is the fast club}$ $\mbox{forcing $\bP_{\fc}$"}$. Let $G_\alpha=G\cap \bP_\alpha$.
Let $\dot{A}\in V[G]$  be an $H$-name for $A$. Choose $\beta$ large enough such that 
$\dot{A} \in V[\langle G_{\alpha}: \alpha < \beta\rangle].$
Now  $G_{\beta}$ is $\bP_\fc$-generic over $V[\langle G_{\alpha}:\alpha < \beta\rangle]$. But, $V[G]$ is a generic extension of $V[\langle G_{\alpha}:\alpha < \beta\rangle][G_{\beta}]$ by countably closed forcing
and by assumption, in $V[G][H]$, there
is a club $D \subseteq C_{G_{\beta}}$ such that $D\cap\eta \in L[A]$ for all $\eta < \omega_1$.   We apply Lemma~\ref{closed} in $V[\langle G_{\alpha}:\alpha < \beta\rangle]$ and conclude that  there is a club 
$D \subseteq C_{G_{\beta}}$  in $V[\langle G_{\alpha}:\alpha < \beta \rangle][G_{\beta}][H]$ such that $D\cap\alpha \in L[A]$ for all $\alpha< \omega_1$. By  Lemma~\ref{code357},  $P(\omega)^V \subseteq L[A]$.
%
\end{proof}

\begin{theorem}\label{aleph1}
There is a set of forcing conditions that forces the existence of a complete non-categorical finite second order theory with a model of cardinality $\aleph_1$.
\end{theorem}

\begin{proof}Assume w.l.o.g., CH. As said above, we start with some preparatory countably closed forcing $\bP$ obtaining a generic extension $V[G]$. Then we add $\aleph_1$ Cohen-reals obtaining a further generic extension $V[G][H]$. In this model we consider for every $x\subseteq\omega$ the model $M_x$ as defined in (\ref{HC}).
Clearly, the cardinality of $M_x$ is $\aleph_1$. 
We shall now show that if $x$ is Cohen-generic over $V[G]$, e.g. one of the $\aleph_1$ many coded by $H$, then the complete second order theory of $M_x$ is  finitely axiomatizable (in second order logic). To end the proof of the theorem, we show that if $x$ and $y$ are mutually Cohen-generic over $V[G]$, then $M_x$ and $M_y$ are second order equivalent but non-isomorphic. 

In order to use second order logic over $\omega_1$ to talk about $HC^V$ and Cohen-genericity over $V$ we need to be able to 
decide, by the means offered by second order logic, which reals in $V[G][H]$ are in $V$ (or, equivalently, in $V[G]$) and which are not. This is precisely the purpose of the preparatory forcing $\bP$.

We denote the starting ground model by $V$ and assume, w.l.o.g., that $V$ satisfies CH. We let the preparatory forcing $\bP=\langle \bP_\alpha:\alpha<\omega_2\rangle$ be the countable support  iteration of $\langle \dot{Q}_\alpha:\alpha<\omega_2\rangle$, where $\bP_\alpha\force ``\dot{Q}_\alpha\mbox{ is the fast club}$ $\mbox{forcing $\bP_{\fc}$"}$. Let $G$ be $\bP$-generic over $V$ and $G_\alpha=G\cap \bP_\alpha$. In $V[G]$ we force with $\bQ$ a generic $H$.
Note that $\aleph_1^{V[G][H]}=\aleph_1^V$ and $\P(\omega)^{V[G]}=\P(\omega)^V$. Working in $V[G][H]$, let the second order sentence $\phi(R,E)$, where $R$ is unary and $E$ is binary, say in a model $M$:
\begin{enumerate}
\item[(1)] $E^M$ is a well-founded relation satisfying $ZFC^-$ $+$ ``every set is countable".   This should be also true when relativized to $R^M$.         
\item[(2)] $|M|=\aleph_1$.
\item[(3)] If $P'\in R^M$ denotes (in $M$) the set $\fn(\omega,2,\omega)$ of conditions for adding one Cohen real, then there is $K\subseteq P'$ such that $K$ is $P'$-generic over $R^M$ and $M\models ``V=R[K]"$.
\item[(4)]  If $a\subseteq\omega$ and the transitive collapse of $M$ is $N$, then the following conditions are equivalent:
\begin{enumerate}
\item $a\in R^N$.
\item If $A\subseteq \omega_1$ and for every club $C\subseteq\omega_1$ there is a club $D\subseteq C$   such that $D\cap\alpha\in L[A]$ for every $\alpha<\omega_1,$ then $a\in L[A]$.

\end{enumerate}\end{enumerate} 

Note that we can express $``D\cap\alpha\in L[A]"$, or equivalently $``\exists\beta(|\beta|=\aleph_1\wedge D\cap\alpha\in L_\beta[A]"$, in second order logic on $M$ since second order logic gives us access to all structures of cardinality $|M|$ (=$\aleph_1$).

\medskip

\noindent{\bf Claim:} The following conditions are equivalent in $V[G][H]$:

\begin{enumerate}
\item[(i)] $M\models\phi(R,E)$.
\item[(ii)] $M\cong M_x$ for some real $x$ which is Cohen generic over $V$. 
\end{enumerate}

\begin{proof}
(i) implies (ii): Suppose $M\models\phi(R,E)$. Let $(N,U,\in)$ be the transitive collapse of $(M,R^M,E^M)$. 
By (3), there is $r$ which is Cohen-generic over $U$ and $N=HC^{U[r]}$. We show that $U=HC^V$. Suppose $a\in\P(\omega)^V$. We use condition (4) to demonstrate that $a\in U$. To this end, let $A$ be as in (4b). By 
Lemma~\ref{eww}, $a\in L[A]$. Thus (4) implies $a\in U$. On the other hand, suppose $a\in (\P(\omega))^U$.
We again use (4) to show that $a\in \P(\omega)^V$. Let $A\subseteq\omega_1$ code $([\omega_1]^\omega)^V$. 
If $C$ is any club in $V[G][H]$, then, since $H$ is obtained by means of a CCC forcing, there is a club $D\subseteq C$ in $V[G]$. Now $D\cap\alpha\in V$, whence $D\cap\alpha\in L[A]$, for all $\alpha<\omega_1$. It follows that $a\in L[A]$. Since $A\in V$, we may conclude $a\in \P(\omega)^V$.  
Hence, $U=HC^V$ and $r$ is Cohen-generic over $V$. We have proved (ii).

(ii) implies (i): Suppose $(N,R^N,E^N)=(HC^{V[r]},HC^V,\in)$, where $r$ is $\fn(\omega,2,\omega)$-generic over $V$. We show that $(N,R^N,E^N)\models\phi(R,E)$. Conditions (1) and (2) are trivially satisfied. Condition (3) holds by construction. To prove that condition (4) holds, suppose $a\subseteq\omega$ and let $A$ be as in (4).  By 
Lemma~\ref{eww}, $a\in L[A]$. Condition (4) and thereby the Claim is proved.
\end{proof}

We continue the proof of Theorem~\ref{aleph1}. The sentence $\phi(R,E)$ is non-categorical in $V[G][H]$ because if we take two mutually generic (over $V[G]$) Cohen reals $r_0$ and $r_1$, then $M_{r_0}$ and $M_{r_1}$ are non-isomorphic models of $\phi(R,E)$. To prove that $\phi(R,E)$ is complete, suppose $(M,R^M,E^M)$ and $(N,R^N,E^N)$ are two models of  $\phi(R,E)$. W.l.o.g., they are of the form $(M,R^M,\in)$ and $(N,R^N,\in)$, where $M$ and $N$ are transitive sets. By construction, they are of the form $M_{r_0}$ and $M_{r_1}$ where both $r_0$ and $r_1$ are Cohen generic over $HC^V$, hence over $HC^{V[G]}$. They are subsumed by the generic $H$. By homogeneity of Cohen forcing $\fn(\omega,2,\omega)$ the models are second order equivalent.
\end{proof}

In fact the forcing gives something stronger. If $\kappa$ is a cardinal that is second order characterizable in the forcing extension, we may replace the model $M_x = (HC^{V[x]}, HC^V, \in)$, where $x \subseteq \omega$ is Cohen over $V$, with the model $(\kappa \cup HC^{V[x]}, HC^V, \in)$, and the proof of Theorem \ref{aleph1} goes through mutatis mutandis:
\begin{corollary}
There is a set of forcing conditions that forces the following: if $\kappa$ is any second order characterizable cardinal, there is a complete non-categorical finitely axiomatizable second order theory with a model of cardinality $\kappa$. \qed
\end{corollary}

Since the non-isomorphic models above derive from mutually generic Cohen reals, it follows that the non-categorical theories in question have (at most) continuum many non-isomorphic models. We lastly mention how to get non-categorical theories with more models than this.

It is straightforward to see that in theorem \ref{aleph1} and the constructions preceding it, the cardinal $\aleph_1$ may be replaced with any cardinal $\mu^+$ with $\mu$ regular. That is, the $\omega_2$-length countable support iteration of fast club forcing at $\omega_1$ is replaced by a $\mu^{++}$-length $\leq \mu$-sized support iteration of fast club forcing at $\mu^+$, and the forcing to add $\aleph_1$ many Cohen subsets of $\omega$ is replaced by adding $\mu^+$ many Cohen subsets of $\mu$. The model $M_x$ is then taken to be of the form $(H(\mu)^{V[x]}, H(\mu)^V, \in)$ where $x$ is a Cohen subset of $\mu$ generic over $V$.

From this variation, we then get the following corollary.
\begin{corollary}
Suppose $\mu$ is a regular cardinal. There is then a set of forcing conditions that forces the following: if $\mu$ is second order characterizable, and if $\kappa \geq \mu$ is any second order characterizable cardinal, there is a complete non-categorical finite second order theory $T$ with a model of cardinality $\kappa$. Also, the theory $T$ has between $\mu^+$ and $2^\mu$ many models up to isomorphism.
\end{corollary}
Note that the concern of the second order characterizability of $\mu$ and $\kappa$ in the forcing extension are irrelevant for cardinals with simple definitions such as $\aleph_n$, $n < \omega$ or $\aleph_{\omega_1 + 1}$, for example.

In conclusion we cannot hope to prove the categoricity of finite complete second order theories from large cardinals even if we restrict to theories which have a model of regular uncountable cardinality.

\section{Forcing categoricity}

In \cite{AFWO1} (for $\kappa > \omega_1$) and \cite{AFWO2} (for $\kappa = \omega_1$), Aspero and Friedman proved the following:
\begin{theorem}
Suppose $\kappa$ is the successor of a regular cardinal, and uncountable. Then there is a poset $\bP$ such that in a generic extension by $\bP$, there is a lightface first order definable well-order of $H(\kappa^+)$.
\end{theorem}

Since we can translate a first order lightface definable well-order of $H(\kappa^+)$ into a well-order of $\powset(\kappa)$ that is second order definable over any structure of cardinality $\kappa$, we obtain the following corollary.

\begin{theorem}
Suppose $\kappa$ is the successor of a regular cardinal, uncountable, and that $\kappa$ is second order characterizable. Then there is a poset $\bP$ that forces the following: every finitely axiomatizable second order theory with a model of cardinality $\kappa$ is categorical. \qed
\end{theorem}

We are thus left to consider the case of theories with models of limit cardinality, whether regular or singular.

The following theorem shows that the categoricity of complete second order theories with a model of singular cardinality is (relatively) consistent with large cardinals. We are indebted to Boban Veli\v{c}kovi\'{c} for suggesting how to improve an earlier weaker version of this result.

\begin{theorem}\label{sing}
Suppose $\kappa$ is a singular strong limit with uncountable cofinality $\lambda$. Then there is a forcing notion $\bP$ of cardinality $\kappa$ such that  
\begin{enumerate}
\item $\bP$ preserves $\kappa$ singular strong limit of uncountable cofinality $\lambda$.
\item $\bP$ forces the statement: Every finitely axiomatizable complete second order theory with a model of cardinality $\kappa$ is categorical.
\end{enumerate}
\end{theorem}
\begin{proof} 
W.l.o.g. we assume GCH up to $\kappa$. We first force a second order definable well-order of the bounded subsets of $\kappa$ with a reverse Easton type iteration of length $\kappa$ described in \cite[Theorem 20]{MR540771}.  

Let $e:\kappa\to\kappa$ be the function which lists the set $B$ of beth fixed points $>\lambda$ in increasing order, and let $S=\langle\kappa_\xi: \xi<\lambda\rangle \subseteq B$ be an increasing cofinal sequence in $\kappa$ such that $\kappa_0>\lambda$.  Let $\pi:\kappa\times\kappa\to\kappa$ be the G\"odel pairing function. Let $W$ be a well-order of $V_\kappa$.
Suppose $A\subseteq\mu$, where $\mu\in B$. We write $A\sim V_\mu$ if 

$$(V_\mu,\in)\cong(\mu,\{(\alpha,\beta):\pi(\alpha,\beta)\in A\}).$$ 

\noindent Let the poset  $E(\mu,A)$ be the iteration (product) of the posets $\bR_\alpha$, $\alpha<\mu$, where
$$\bR_\alpha=
\left\{\begin{array}{ll}
\fn(\aleph_{\mu+\alpha+3}\times \aleph_{\mu+\alpha+1},2,\aleph_{\mu+\alpha+1}),
&\mbox{ if $\alpha=\omega\cdot\beta$ and $\beta\in A$}\\
\fn(\aleph_{\mu+\alpha+4}\times \aleph_{\mu+\alpha+2},2,\aleph_{\mu+\alpha+2}),
&\mbox{ if $\alpha=\omega\cdot\kappa_\xi+1$, $\xi<\lambda$}\\
(\{0\},=)&\mbox{ otherwise}\end{array}\right.$$
\noindent with Easton support i.e. $E(\mu,A)$ consists of functions 
$p\in \prod_{\alpha<\mu} \bR_\alpha$ such that, denoting the support $\{\alpha:f(\alpha)\ne\emptyset\}$ of $f$ by $\supp(p)$,  $|\supp(p)\cap\gamma|<\gamma$ for all regular $\gamma$. 
 
We now define an iteration $\langle\bP_\alpha:\alpha<\kappa\rangle$ with the property that $\bP_\alpha$ does not change beth fixed points $\beta=\beth_\beta$ for any $\beta$. We let $\bP=\langle \bP_\alpha:\alpha<\kappa\rangle$ be the following iteration:
If $\alpha$ is a limit ordinal, we use direct limits for regular $\alpha$ and inverse limits for singular $\alpha$. Suppose then $\alpha=\beta+1$. Let $\dot{A}$ be the $W$-first $\bP_\beta$-name $\dot{A}$ in $V_\kappa$ such that $\bP_\beta\force \dot{A}\sim V_{\check{e}(\check{\beta})}$. Then $\bP_\alpha=\bP_\beta\star E(\check{e}(\check{\beta}),\dot{A})$. Let $G$ be $\bP$-generic over $V$ and $G_\alpha=G\cap \bP_\alpha$.


In the forcing extension $V[G]$, for every $\mu\in B$ there is a set $A\subseteq\mu$ which codes, via the canonical bijection $\pi:\kappa\times\kappa\to\kappa$, a bijection $f_A \colon \mu\to (V_\mu)^{V[G]}$. The set $A$ itself satisfies
$$V[G]\models A=\{\alpha<\mu: 2^{\aleph_{\mu+\omega\cdot\alpha+1}}=\aleph_{\mu+\omega\cdot\alpha+3}\}$$
and from $A$ we can read off $f_A$ and a well-order $<_\mu^*$ of $(V_\mu)^{V[G]}$:
$$V[G]\models f_A(\alpha)<_\mu^*f_A(\beta)\iff \alpha<\beta<\mu.$$

Now working in $V[G]$, fix a collection $\mathcal F \subseteq \powset(\kappa)$, and we set out to define a well-order not on the whole of $\mathcal F$ but a certain subset of it. Define a relation $R$ on $\mathcal F$ by
\[
X R Y \iff X \cap \kappa_\xi <^*_{\kappa_\xi} Y \cap \kappa_\xi \text{ for all but boundedly many } \xi < \lambda.
\]
As $\lambda$ is uncountable, $R$ is well-founded, so the set
\[
\mathcal W = \set{X \in \mathcal F : X \text{ is minimal in } R}
\]
is nonempty, and if $X, Y \in \mathcal W$ with $X \neq Y$, then both $X \cap \kappa_\xi <^*_{\kappa_\xi} Y \cap \kappa_\xi$ and $Y \cap \kappa_\xi <^*_{\kappa_\xi} X \cap \kappa_\xi$ occur for unboundedly many $\xi < \lambda$.

To see that $\abs{\mathcal W} < \kappa$, suppose to the contrary that $\abs{\mathcal W} \geq \kappa$ and define a coloring $c \colon [\mathcal W]^2 \to \lambda$ by $c(\set{X,Y}) = \pi(\xi_1, \xi_2)$ where $\xi_1$ is the least $\xi < \lambda$ such that $X \cap \kappa_\xi <^*_{\kappa_\xi} Y \cap \kappa_\xi$, and $\xi_2$ is the least $\xi < \lambda$ such that $Y \cap \kappa_\xi <^*_{\kappa_\xi} X \cap \kappa_\xi$. Since $\abs{\mathcal W} \geq \kappa > (2^\lambda)^+$, by the Erdös-Rado theorem there is a set $H \subseteq \mathcal W$ homogeneous for $c$ of color $\pi(\xi_1, \xi_2)$ and cardinality $\lambda^+$. But this is a contradiction, since ordering $H$ in $<^*_{\kappa_{\xi_1}}$-increasing order yields an infinite decreasing sequence in the well-order  $<^*_{\kappa_{\xi_2}}$, so $\abs{\mathcal W} < \kappa$.

Now for each $X \in \mathcal W$, define $f_Y \colon \lambda \to \kappa$ such that $f_X(\xi)$ is the index of $X \cap \kappa_\xi$ in the well-order $<^*_{\kappa_\xi}$. Then the set $\bigcup \set{\ran(f_X) : X \in \mathcal W}$ has some cardinality $\gamma < \kappa$, and we can let $h \colon \bigcup \set{\ran(f_X) : X \in \mathcal W} \to \gamma$ be the transitive collapse map.

Then for $X \in \mathcal W$, the function $h \circ f_X \colon \lambda \to \gamma$ can be encoded as a subset of a large enough $\mu \in B$, and obviously $h \circ f_X \neq h \circ f_Y$ if $X \neq Y$, so we can well-order $\mathcal W$ by
\[
X \lhd Y \iff h \circ f_X <^*_\mu h \circ f_Y
\]
and all this is second order definable in $V[G]$ in a structure of size $\kappa$, if the collection $\mathcal F$ is. This allows us to pick a distinguished element of $\mathcal F$ as the $\lhd$-least $R$-minimal element.

Suppose now that $\phi$ is a complete second order sentence with a model $M$ of cardinality $\kappa$, and let $\mathcal F$ consist of the set of $X \subseteq \kappa$ encoding a model of $\phi$. Note that over a model of cardinality $\kappa$ we can write a formula $\phi_R(X,Y)$ expressing $XRY$ for $X, Y \in \mathcal F$, a formula $\phi_{\mathcal W}(X)$ expressing $X \in \mathcal W$, and a formula $\phi_\lhd(X,Y)$ expressing $X \lhd Y$ if $X$ and $Y$ are $R$-minimal.

Let $M \models \Phi$ now say that $X \subseteq M$ encodes a model isomorphic to $M$ (and thus satisfies $\phi$), and for any $Y \subseteq M$ that also encodes a model of $\phi$, $\lnot \phi_R(Y,X)$, and moreover if for all $Z \subseteq M$ that encode a model of $\phi$ also $\lnot \phi_R(Z,Y)$, then $X = Y$ or $\phi_\lhd(X,Y)$. That is, $X \in \mathcal W$ and if also $Y \in \mathcal W$ then $X = Y$ or $X \lhd Y$, which uniquely specifies $X$. As the model of $\phi$ with the least code in this sense satisfies $\Phi$ and $\phi$ is complete, $\phi$ implies $\Phi$ and thus that all models of $\phi$ are isomorphic, so $\phi$ is categorical.
\end{proof}

The method of the preceding proof does not extend to the cases of the limit cardinal $\kappa$ being regular, or of countable cofinality, so these cases are left open.

In conclusion, no known large cardinal axiom (e.g. the existence  of huge cardinals)
can decide whether all complete second order theories with a model of singular cardinality are categorical. In particular, such axioms cannot imply that all finite complete second order theories are categorical.

\section{Theories with only countably many models}

Since under PD we have non-categorical complete recursively axiomatized second order theories, we may ask how badly categoricity can fail in those cases? Echoing  Vaught's Conjecture, we may ask whether the number of countable non-isomorphic models of a complete recursively axiomatized second order theory is always countable or $2^\omega$. Leaving this question unresolved, we have the following result which demonstrates the ability of categorical theories to `capture' (in the sense of \cite{MR0221924}) the models of non-categorical theories.
  
\begin{theorem}\label{three} Assume $AD^{L(\R)}$.
If $T$ is a recursively axiomatized complete second order theory with only countably many non-isomorphic countable models, then there is a recursively axiomatized categorical second order theory $S$ the unique model of which interprets all the countable models of $T$.\end{theorem}

\begin{proof}
Let $T$ be a recursively axiomatized second order theory with only countably many non-isomorphic countable models.  Let $A$ be the $\Pi^1_\omega$ (i.e. an intersection of a recursively coded family of sets each of which is $\Pi^1_n$ for some $n$) set of reals that code a model of $T$. Since $A$ is a countable union of equivalence classes of the $\Sigma^1_1$-equivalence relation of isomorphism, we may conclude that $A$ is $\bSigma^1_1$.

We wish to show that $A$ is $\Pi^1_2(r_0)$ in a parameter $r_0$ which is a $\Pi^1_\omega$ singleton.
For this, we mimic a proof of Louveau (Theorem 1 in \cite{MR757030}) to show:
\begin{theorem}\label{lou}
Assume $AD^{L(\R)}$. Every  $\bSigma^1_1$ set which is $\Pi^1_\omega$ is $\Pi^1_2(r_0)$ for some real $r_0$ such that  $\{r_0\}$ is a $\Delta^1_{\omega+1}$-singleton.
\end{theorem}
\begin{proof}
Let $A$ be a $\bSigma^1_1$ set that is also $\Pi^1_\omega$, say $A = \bigcap_n A_n$ with each $A_n$ being $\Pi^1_n$. Let also $U \subseteq (\omega^\omega)^2$ be a universal $\Sigma^1_1$ set.

We define for each $n$ a game $G_n$ on $\omega$ where players I and II take turns to play the digits of reals $\alpha$ and $\gamma$ respectively (there is no need to let II pass turns here). Then II wins a play of $G_n$ if $\alpha \in A \implies \gamma \in U$ and $\alpha \notin A_n \implies \gamma \notin U$. 
\[
\begin{array}{c|ccccc}
\text{I} & n_0 & & n_1 & & \cdots \\
\hline
\text{II} & & m_0 & & m_1 & \cdots
\end{array}
\quad
\begin{matrix}
\alpha \\ \gamma
\end{matrix}
\]
As in Louveau's proof, II has a winning strategy as follows: since $A$ is $\bSigma^1_1$, we have $A(x) \iff U(y,x)$ for some $y$, so II wins by playing the digits of $\langle y, \alpha \rangle$ (as I is playing the digits of $\alpha$). The complexity of the winning set for II in $G_n$ is $\Sigma^1_\omega$, so by Moschovakis's strategic basis theorem (\cite{MR2526093}, Theorem 6E.2), II has a winning strategy $\sigma_n$ that is a $\Delta^1_{\omega+1}$-singleton. Note that the pointclass $\Sigma^1_\omega$, i.e. the collection of countable unions of recursively coded families of projective sets, is both adequate and scaled (see Remark 2.2 in \cite{rudominer}, essentially Theorem 2.1 in \cite{MR0730592}).

Then the set $B_n = \set{y \mid (y * \sigma_n)_{\text{II}} \in U}$ is a $\Sigma^1_1(\sigma_n)$ set with $A \subseteq B_n \subseteq A_n$ (where $(y * \sigma_n)_{\text{II}}$ denotes the real $\gamma$ the strategy $\sigma_n$ produces as I plays $\alpha = y$), so altogether $A = \bigcap_n B_n$ is a $\Pi^1_2(s_0)$ set where $s_0 = \langle \sigma_n \mid n < \omega \rangle$ is a $\Delta^1_{\omega+1}$-singleton.
\end{proof}

We may reduce the complexity of the parameter down to being a $\Pi^1_\omega$ singleton by the following theorem of Rudominer:
\begin{theorem}[Rudominer \cite{rudominer}]\label{rudo}
Assume $AD^{L(\R)}$. Then every real $s_0$ which is a $\Sigma^1_{\omega+1}$ singleton, is recursive in a real $r_0$ which is a $\Pi^1_\omega$ singleton. \qed
\end{theorem}
Therefore the set $A$ is a $\Pi^1_2(r_0)$ set where $r_0$ is a $\Pi^1_\omega$ singleton. Let $\eta(r,s)$ be a second order $\Pi^1_2$ formula which defines the predicate $s\in A$ on $(\N,+,\times,r_0)$. Let $\theta_1(Q_+,Q_{\times})$ be the standard second order characterization of $(\N,+,\times)$, as above in the proof of Theorem~\ref{one}. Let  
$\psi_n(Q_+,Q_{\times},s)$, $n<\omega$, be second order formulas such that if $X_n$ is the set of reals $s$ satisfying  $\psi_n(Q_+,Q_{\times},s)$ in $(\N,+,\times)$, then $\{r_0\}=\bigcap_nX_n$. Let  $P$ be a new unary predicate symbol and

$$S=\{\theta_1(Q_+,Q_{\times})\}\cup\{\psi_n(Q_+,Q_{\times},P) : n<\omega\}.$$

Suppose $M$ is a model of $S$. W.l.o.g. the arithmetic part of $M$ consists of the standard $+$ and $\times$ on $\N$. Let $s$ be the interpretation of $P$ in $M$. Then $s=r_0$. Thus $S$ is categorical.  The theory $S$  is recursive because the proofs of Theorems~\ref{lou}~and~\ref{rudo} are sufficiently uniform. In conclusion, $M$ is categorically characterized by the recursive second order theory $S$.

Now the countable models of $T$ are interpretable in $S$ in the following sense: a real $s$ codes a model of $T$ if and only if $M\models\eta(r_0,s)$. We also get a translation of sentences: if $\phi$ is a second-order sentence in the vocabulary of $T$, letting $\hat \phi$ be the sentence $\exists X (\eta(r_0,X) \land X \models \phi)$, we have that $\phi \in T$ if and only if $\hat \phi \in S$.
\end{proof}

\section{Definable models of categorical theories}

Suppose we are given a categorical second order theory $T$. Naturally, we assume that $T$ has a model, otherwise categoricity is vacuous. But what can be said about the models of $T$ apart from their isomorphism with each other? In particular, can we always find a model which is definable in some reasonable sense, e.g. hereditarily ordinal definable?
To emphasize this point, consider the  second order sentence which characterizes  the structure $(\N,+,\cdot,0^\sharp)$.
This categorical sentence  has  no models in $L$. We ask, can we have a categorical sentence with no models in $\hod$? Since it could be that $V=\hod$, we are looking at this question under assumptions stronger than ZFC.

The following result of Kaplan and Shelah is useful for us:

\begin{theorem}[\cite{MR3601093}]\label{ks}
 If $\bP$ forces the collapse of $|\omega_2|$ to $\omega$, then
there is a $\bP$-term $\tau$ for a countable model such that
\begin{enumerate}
\item   If $G_1\times G_2$ is generic for $\bP \times \bP$ then 
$$V[G_1][G_2] \models M_1\cong M_2,$$
where $M_1$ is the interpretation $\tau^{G_1}$ of $\tau$ by $G_1$ and $M_2$ is  $\tau^{G_2}$.

\item  $\bP \force ``\tau$ is not isomorphic to $\check{M}$", for any $M$ in $V$.

\end{enumerate}

\end{theorem}

We make some observations about the proof. It involves a construction of Laskowski and Shelah:
\begin{theorem}[\cite{LasShe}]
There is a countable consistent first order theory $T$, with a predicate $V$ in its vocabulary, having the following property. For any model $M \models T$ and any $A \subseteq V^M$, isolated types are dense over $A$ but the theory $T(A) = \Th(M,a)_{a \in A}$ has an atomic model if and only if $\abs{A} < \omega_2$.
\end{theorem}

The theory $T$ is as follows. Let $L$ be a countable vocabulary consisting of two unary predicates $U, V$, one unary function symbol $p$, as well as binary relations $R_n$ and binary functions $f_n$ for $n < \omega$ (the functions will not be total, but instead have domain $U$). Let $K$ be the collection of all finite $L$-structures satisfying a certain finite list of first order axioms (see \cite{LasShe}). Let $\mathcal B$ be the Fraïsse limit of $K$ and let $T = \Th(\mathcal B)$. The theory $T$ is well defined since $\mathcal B$ is unique up to isomorphism.

We then form an uncountable model of the theory $T$ as follows. For an ordinal $\alpha$ let $L_\alpha$ be the vocabulary $L$ together with $\alpha$ many new constant symbols $c_\beta$, $\beta <\alpha$. Using a standard Henkin construction, we form a term model for the theory $T$ together with the additional axioms stating that the new constant symbols name distinct elements.  We let $T(A_\alpha)$ be the theory of this term model in the vocabulary $L_\alpha$. (Although the Henkin construction involves forming the completion of a theory, we can make the choice of which completion to use definable by referring to the well-ordering of the sentences.)

We can also observe that for a countable ordinal $\alpha$, the class of countable atomic models of $T(A_\alpha)$ is definable from $T(A_\alpha)$, which itself is definable from $\alpha$, and the definitions can be carried out in $H(\omega_1)$. Using these two observations, the following obtains:
 

\begin{theorem}[ZF] \label{nothod}
Assume $\omega_2^{\scriptsize\hod}$ is countable. Then there is a countable model M such that
\begin{enumerate}
\item The isomorphism class of M is ordinal definable.
\item  There is no model in $\hod$ which is isomorphic to M.
\end{enumerate}
Moreover, if the property of a linear order of being of order-type $\omega_2^{\scriptsize \hod}$ is second order definable in the countably infinite structure of the empty vocabulary, then the second order theory of $M$ is finitely axiomatizable.
\end{theorem}

\begin{proof}
Let $\alpha = \omega_2^{\text{HOD}}$. Let $T(A_\alpha)$ be the theory constructed above. Finally,  let $M$ be a countable atomic model of $T(A_\alpha)$. Since $\hod$ satisfies  $\abs{T(A_\alpha)} = \omega_2$, the theory $T(A_\alpha)$ has no atomic model in $\hod$, but as being an atomic model is absolute, this shows that there is no model in $\hod$ isomorphic to $M$.

The isomorphism class of $M$ is ordinal definable as the class of countable atomic models of $T(A_\alpha)$, which is definable from $\alpha$. Additionally, if $\alpha$ is second order definable in the countably infinite structure of the empty vocabulary, we can define the theories $T$ and $T(A_\alpha)$ in second order logic  expressing ``I am isomorphic to a countable atomic model of $T(A_\alpha)$'' with a single second order sentence. This finitely axiomatizes the second order theory of $M$.
%
%
\end{proof}

Of course, the assumption that $\omega_2^{\scriptsize \hod}$ is second order definable in the countably infinite structure of the empty vocabulary is somewhat ad hoc. However, it holds, for example, in $L[G]$, where $G$ is $P$-generic over $L$ for $P = \Coll(\omega, <\omega_3)^L$.
This is because the poset $P$ is weakly homogeneous, so $\hod^{L[G]} = \hod^L(P) = L$, whence $\omega_2^{\scriptsize \hod} = \omega_2^L$ is countable and second order definable in any countable model in $L[G]$.

\medskip


We also obtain the following variation:

\begin{corollary}\label{cor}
Assume  $ZFC + AD^{L(\R)} + ``\hod \hspace{2pt}\cap\hspace{2pt} \R = \hod^{L(\R)} \cap\hspace{2pt} \R"$ and that $\omega_2^{\scriptsize \hod}$ is definable in $\hod^{L(\R)} \restriction \Theta^{L(\R)}$ and countable. Let $M$ be the countable model of Theorem \ref{nothod}. Let $N = (\Theta^{L(\R)},<,M)$ (w.l.o.g. the domain of $M$ is $\omega$).  Then the second order theory of $N$ is finitely axiomatizable and categorical but has no model which belongs to $\hod$.
\end{corollary}

%

\begin{proof}
We can use  \cite[Theorem 3.10, Chapter 23]{handbook})
to define $\hod^{L(\R)} \restriction \Theta^{L(\R)}$ and $L_{\Theta^{L(\R)}}(\R)$ from $\Theta^{L(\R)}$ in second order logic, which then allows us to define $\omega_2^{\scriptsize \hod}$ and $M$ as in Theorem \ref{nothod}. 
\end{proof}

The assumptions of Corollary~\ref{cor} follow, for example, from $ZFC + AD^{L(\R)} + V = L(\R)[G]$, where $G$ is $\Pmax$-generic, as then $\hod^{L(\R)} = \hod^{L(\R)[G]}$ and $\omega_2^{\scriptsize \hod}$ is countable.

\section{Open questions}

%
The following question was raised by Solovay \cite{Solo}: 

\begin{problem}
Assuming $V=L$, is every recursively axiomatized complete second order theory categorical? 
\end{problem}

Our results  do not solve this  one way or another, and it remains an interesting open question. In $L[U]$ there are recursively axiomatized complete non-categorical second order theories, but we do not know if such theories necessarily have only large models:

\begin{problem}
Suppose $V = L[U]$, $\kappa$ is the sole measurable cardinal of $L[U]$, and $T$ is a complete recursively axiomatized second order theory that has a model of cardinality $\lambda < \kappa$ such that $\lambda$ is second order characterizable. Is $T$ categorical?
\end{problem}

There are many other open questions related to finite or recursively axiomatized complete second order theories with uncountable models. We showed that we can force categoricity for successor cardinals of regular cardinals, and some singular limit cardinals, but the following two cases were left open:

\begin{problem}
Can we always force the categoricity of all finite complete second order theories with a model of cardinality $\kappa$, where $\kappa$ is either a regular (non-measurable) limit cardinal, or singular of cofinality $\omega$?
\end{problem}


An $I_0$-\emph{cardinal} is a cardinal $\lambda$ such that there is $j \colon L(V_{\lambda+1})\to L(V_{\lambda+1})$ with critical point below $\lambda$. Note that then $\lambda$ is singular of cofinality $\omega$, $\lambda^+$ is measurable in $L(V_{\lambda+1})$ (\cite{extender}), and the Axiom of Choice fails in $L(V_{\lambda+1})$ (\cite{MR0311478}). This is in sharp contrast to the result of Shelah that if $\lambda$ is a singular strong limit cardinal of uncountable cofinality, then $L(\P(\lambda))$ satisfies the Axiom of Choice (\cite{MR1462202}).
Since Axiom of Choice fails in $L(V_{\lambda+1})$, there can be no  well-order of $\P(\lambda)$ which is second order definable on $\lambda$. This raises the following question:

\begin{problem}
Is every finite complete second order theory with a model of cardinality of an $I_0$-cardinal categorical (or, at least categorical among all models of that cardinality)?
\end{problem}


\end{document}